\documentclass[a4,dvipdfmx,10pt]{article}
\usepackage{amsmath,amssymb,amsthm}
\usepackage{url}
\usepackage{enumitem}

\newtheorem{my-theorem}{Theorem}[section]
\newtheorem{my-lemma}[my-theorem]{Lemma}
\newtheorem{my-proposition}[my-theorem]{Proposition}
\newtheorem{my-corollary}[my-theorem]{Corollary}
\newtheorem{my-definition}[my-theorem]{Definition}
\newtheorem{my-puzzle}[my-theorem]{Puzzle}
\newtheorem{my-question}[my-theorem]{Question}

\newenvironment{theorem}[1][]{
  \begin{my-theorem}[#1]%
  \pushQED{\qed}}%
  {\popQED\end{my-theorem}
}
\newenvironment{lemma}[1][]{
  \begin{my-lemma}[#1]%
  \pushQED{\qed}}%
  {\popQED\end{my-lemma}
}
\newenvironment{proposition}[1][]{
  \begin{my-proposition}[#1]%
  \pushQED{\qed}}%
  {\popQED\end{my-proposition}
}
\newenvironment{corollary}[1][]{
  \begin{my-corollary}[#1]%
  \pushQED{\qed}}%
  {\popQED\end{my-corollary}
}
\newenvironment{definition}[1][]{
  \begin{my-definition}[#1]%
  \pushQED{\qed}}%
  {\popQED\end{my-definition}
}
\newenvironment{puzzle}[1][]{
  \begin{my-puzzle}[#1]%
  \pushQED{\qed}}%
  {\popQED\end{my-puzzle}
}
\newenvironment{question}[1][]{
  \begin{my-question}[#1]%
  \pushQED{\qed}}%
  {\popQED\end{my-question}
}

\renewcommand{\proofname}{Proof}

\makeatletter
\renewenvironment{proof}[1][\proofname]{\par
  \pushQED{\qed}%
  \normalfont \topsep6\p@\@plus6\p@\relax
  \trivlist
  \item[\hskip\labelsep
        \itshape
    {\bf\underline{#1}\@addpunct{\ }}]\ignorespaces
}{%
  \popQED\endtrivlist\@endpefalse
}
\makeatother

\newcommand{\nee}[1]{#1 ^{\neq}}
\newcommand{\dash}[1]{#1 ^{\prime}}
\newcommand{\op}[2]{\langle #1,#2 \rangle}
\newcommand{\map}[3]{#1\colon #2 \to #3}
\newcommand{\cv}[3]{#1\lfloor #2|#3 \rfloor}
\DeclareMathOperator{\dom}{dom}
\DeclareMathOperator{\ran}{ran}
\newcommand{\restric}{\mbox{$\upharpoonright$}}
\newcommand{\MapSet}[2]{\mbox{}^{#1}#2}
\newcommand{\Match}[2]{\operatorname{Match}(#1,#2)}
\newcommand{\CanSee}[1]{\raisebox{-0.4ex}{\mbox{$\overrightarrow{#1}$}}}
\newcommand{\SubPrSet}[1]{{A}^{[#1]}}
\newcommand{\HatGame}[4]{\langle #1,#2,#3,#4 \rangle}
\newcommand{\strategy}[1]{\overline{\sigma_{#1}}}
\newcommand{\PS}[4]{
  \mathcal{P}( #1, \mathrm{#2} #3 #4 )
}
\newcommand{\FEPS}[1]{
  \mathcal{P}( #1, \text{finite error} )
}
\newcommand{\robustPS}[1]{
  \mathcal{P}( #1, \text{robust} )
}

\title{A General Theorem for Non-Simultaneous Hat Guessing Puzzles}
\author{Souji Shizuma}
\date{\today}

\begin{document}

\maketitle

\begin{abstract}
  The prisoners and hats puzzle, or simply the hat puzzle, 
  is a family of games in which a group of prisoners are each assigned a colored hat and are asked to guess the color of their own hat.
  Various versions of the puzzle arise depending on the number of prisoners, 
  the number of possible hat colors, and the information available to them before and after the game begins.
  These puzzles are broadly classified according to whether the prisoners’ declarations are made simultaneously or non-simultaneously.
  In this paper we present a general theorem concerning the existence of a winning strategy when the declarations are non-simultaneous.
  We also discuss the relationship between the construction of such strategies and the Axiom of Choice, 
  as well as their connection to the simultaneous-declaration case.
\end{abstract}

\section{Introduction}
\label{section:Introduction}

``Prisoners and hat puzzles'' are famous puzzles 
in which a jailer puts colored hats on prisoners 
and each prisoner must guess the color of their own hat, which they cannot see. 
In this paper we simply call them ``hat puzzles''. 
There are many mathematical studies 
on hat puzzles generalized to infinitely many prisoners, 
mainly in axiomatic set theory, and
a textbook on these infinite hat puzzles was published in 2013 \cite{Hat001}.
First we consider hat puzzles 
in which prisoners declare their guesses simultaneously.
For instance, 
consider the following puzzle:

\begin{puzzle}[2-person, 2-color, simultaneous 
declaration]
  \label{puz:22VcompleteS}
  A jailer places two prisoners in a room 
  and puts a hat on each prisoner.
  Each hat is colored either black or white.
  Neither prisoner can see the color of their own hat,
  but each sees the other's hat.
  Furthermore,
  after they are placed in the room,
  they can have no communication at all.
  Under these conditions, 
  both prisoners must simultaneously declare their guesses,
  either ``black'' or ``white''.
  If at least one prisoner guesses correctly,
  both prisoners win and are released;
  otherwise, 
  both lose and are executed.
  Naturally prisoners do not know how the jailer will place the hats.
  However,
  before they are brought into the room,
  the jailer informs them of all the rules 
  and prisoners can confer a strategy in advance.
  How can prisoners win this game?
\end{puzzle}

A winning strategy
 (which ensures that at least one prisoner guesses correctly) 
will be found in Theorem \ref{thm:FFVcompleteS}.

Here we characterize a hat puzzle
by the following five parameters:
\begin{itemize}
  \item[(E1)] The number of prisoners
  \item[(E2)] The number of colors
  
  \item[(E3)] How each prisoner can see other prisoners' hats
    (a visibility graph)
  
  \item[(E4)] How and when the prisoners declare their guesses
  \item[(E5)] The winning condition for prisoners
\end{itemize}
In Puzzle \ref{puz:22VcompleteS}, 
these parameters are instantiated as follows:
\begin{itemize}
  \item[(E1)] There are 2 prisoners.
  \item[(E2)] There are 2 colors.
  \item[(E3)] Each prisoner sees all hats other than their own 
  	(i.e., each sees the other's hat).
  \item[(E4)] The prisoners declare simultaneously.
  \item[(E5)] At least one prisoner must guess correctly.
\end{itemize}

By varying these parameters,
we obtain different hat games,
and the main question in hat puzzle 
research is to determine whether a strategy exists 
ensuring the given winning condition 
for all possible hat assignments by the jailer.

By considering puzzles
in which prisoners declare 
non-simultaneously,
we obtain a variety of formats.
Here is a relatively simple example:

\begin{puzzle}[5-person, 2-color, non-simultaneous declaration, first type]
  \label{puz:52VcompleteA}
  Modify Puzzle \ref{puz:22VcompleteS} 
  by changing the following parameters:
  \begin{itemize}
    \item[(E1)] There are now 5 prisoners.
    \item[(E4)] 
      One prisoner declares first, 
      and after hearing that declaration,
      the other four prisoners declare simultaneously.
    \item[(E5)] At least 4 prisoners must be correct.
  \end{itemize}
  How can prisoners win this game?
\end{puzzle}

In this case the first prisoner can serve as a ``hint giver'',
enabling the other four to guess their own hats correctly,
thereby achieving the winning condition.
A more concrete version of this puzzle 
is found in Lemma \ref{lem:FGVA}.

Another non-simultaneous declaration 
puzzle is:

\begin{puzzle}[5-person, 2-color, non-simultaneous 
declaration, second type]
  \label{puz:52VstairsA}
  Modify Puzzle \ref{puz:52VcompleteA} as follows:
  \begin{itemize}
    \item[(E3)] 
      The jailer numbers the prisoners 
      $0$ through $4$ and arranges them in a single file.
      Each prisoner sees only the hats of prisoners 
      with a larger number.
    \item[(E4)] 
      The prisoners declare one by one 
      in ascending order of their numbers.
  \end{itemize}
  Again,
  how can prisoners win this game?
\end{puzzle}

In Puzzle \ref{puz:52VstairsA},
each prisoner sees fewer hats than in Puzzle \ref{puz:52VcompleteA},
but they hear more of the previous guesses. 
For instance,
prisoner 2 can no longer see the hat of prisoner 1,
but can hear what prisoner 1 says. 
If prisoner 1 is indeed correct,
prisoner 2 can work off that guess as a clue,
continuing in the same manner as Puzzle \ref{puz:52VcompleteA}.
A more concrete version of this solution 
is found in Lemma \ref{lem:FGVA}.

We can change (E4) in numerous ways according to the number of prisoners.
If the number of prisoners is infinite,
the variety of speaking orders is even larger.
For instance,
with countably many prisoners,
one can set up puzzles in which they declare 
in several distinct groupings,
or speak in a single file,
or speak in blocks of size $2^n$ repeated countably many times, 
and so on.

In fact,
there is related research concerning countably many prisoners 
in a non-simultaneous declaration puzzle 
like Puzzle \ref{puz:52VcompleteA} or Puzzle \ref{puz:52VstairsA}.
For example, Theorem 3.3.2 of \cite{Hat001} 
deals with
a version of Puzzle \ref{puz:52VcompleteA} 
extended to countably many prisoners,
\cite{Hat002} studies 
a countably infinite extension of the scenario 
in Puzzle \ref{puz:52VstairsA},
and \cite{Hat006} takes an even more general approach 
to such countably infinite puzzles,
including the case of three or more colors.

This paper presents 
a general theorem about non-simultaneous puzzles 
that covers a wider variety of declaration orders
than in Puzzle \ref{puz:52VcompleteA} 
or Puzzle \ref{puz:52VstairsA},
along with the extension to infinitely many prisoners.
We also investigate the connection with the 
simultaneous declaration case,
similarly to what has been done in previous research,
while extending the number of prisoners and colors 
in a general manner.

In Section \ref{section:Formalization} 
we formalize these puzzles 
from a mathematical perspective, primarily employing set theory.
In Section \ref{section:PreviousResearch} 
we review previous results 
on simultaneous declaration
hat puzzles (as featured in \cite{Hat001}) 
for both finite and infinite cases, 
including discussions on the role of the Axiom of Choice 
in the infinite case.
In Section \ref{section:FinitePrisoner} 
we focus on 
non-simultaneous declaration puzzles 
with a finite number of prisoners 
and prove some general theorems.
In Section \ref{section:InfinitePrisoner} 
we discuss
non-simultaneous declaration puzzles 
with infinitely many prisoners 
and establish general theorems in parallel to the simultaneous case;
here, again, the Axiom of Choice will come into play.
Finally,
in Section \ref{section:TwoTypeInInfinitePuzzle}
we investigate the relationship between general theorems
for the simultaneous and non-simultaneous cases,
showing how the existence of a robust finite error predictor
can be used to construct a parity function without invoking the Axiom of Choice.

\section{Puzzle formalization}
\label{section:Formalization}

In this section, we introduce  
formal tools necessary to discuss hat puzzles 
using sets and functions,
a style primarily inspired by \cite{Hat001}. 
We begin with some notation for sets and functions.

For a set $X$ and an element $x \in X$,
$\nee{x}$ abbreviates
``some element in $X$ distinct from $x$''.
Let $|X|$ denote the cardinality of $X$.
We use $\op{a}{b}$ to denote the ordered pair of $a$ and $b$.
In this paper,
directed graphs and functions 
are expressed as sets of ordered pairs.
We regard a directed graph $G$ on the set $X$
as $G \subseteq X^2$.

We write $\map{f}{X}{Y}$ to mean $f$ is a function from $X$ to $Y$.
In set-theoretic notation,
$\map{f}{X}{Y}$ means $f \subseteq X\times Y$ and for each $x \in X$,
there is a unique $y \in Y$ such that $\op{x}{y} \in f$.
If $\op{x}{y} \in f$,
we write $f(x) = y$.
The set of all functions from $X$ to $Y$ 
is denoted by $\MapSet{X}{Y}$.
For $\map{f}{X}{Y}$, $x\in X$ and an element $a$,
let $\cv{f}{x}{a}$  be the function from $X$ to $Y\cup \{a\}$ obtained by 
changing the value of $f(x)$ to $a$. 
Formally,

\[
  \cv{f}{x}{a}
  := \Bigl( f \setminus \bigl\{ \op{x}{f(x)} \bigr\} \Bigr)
  \cup \bigl\{ \op{x}{a} \bigr\}.
\]

We say $g$ is an extension of $f$ when $f \subseteq g$,
that is,
$\dom(f)\subseteq \dom(g)$ and for all $x\in \dom(f)$, $f(x)=g(x)$.
Let $f\restric \dash{X}$ denote the restriction of $f$
to $\dash{X}\subseteq \dom(f)$ 
that is,
$f\restric \dash{X} := f \cap (\dash{X} \times \ran(f))$.

Now we introduce formal notation for hat puzzles.
Let $A$ be the set of prisoners and $K$ the set of colors,
with $2 \leq |A|,|K|$.
An element $f \in \MapSet{A}{K}$ is called a coloring, 
representing an assignment of
hat colors to all prisoners. 

A directed graph $V \subseteq A^2$ 
without loops 
(that is, 
$\op{a}{a} \notin V$ for all $a\in A$) 
is called a visibility graph on $A$,
where $\op{a}{b} \in V$ indicates 
that prisoner $a$ sees prisoner $b$'s hat.
We write $a\CanSee{V} b$ for $\op{a}{b}\in V$.

$V(a)$ denotes the set of prisoners that $a$ sees 
(that is,
$V(a) = \{b\mid \op{a}{b}\in V\}$).
Note that, if
two colorings $f,g$ satisfy $f\restric V(a)=g\restric V(a)$,
prisoner $a$ cannot distinguish those two.

Let $A$ be a set of prisoners,
and $V$ a visibility graph on $A$.
Let $\alpha$ be an ordinal and $I$ be 
a surjection from $A$ onto $\alpha$.
We call $I$ an inning function on $A$.
Let $IN=\ran(I)$.
The ordinal $IN$ represents the number of innings,
where all declarations made by prisoners at the same time are counted as one inning.
For each $\beta\in IN$, define
\[
  \SubPrSet{\beta}
  :=
  \{a\in A\mid I(a)=\beta\}.
\]
Hence $\SubPrSet{0}$ is the set of prisoners
who declare in the zeroth inning,
and, in general, $\SubPrSet{\beta}$ is the set
of prisoners who declare in the $\beta$-th inning.
Let $H(a) := \{ b \mid I(b)<I(a)\}$.
$H(a)$ is the set of prisoners 
whose declaration $a$ can hear.
For $\beta,\gamma\in IN$ with $\beta\leq\gamma$,
we use $\SubPrSet{\beta-\gamma}$ to abbreviate
$\SubPrSet{\beta}\cup\dots\cup\SubPrSet{\gamma}$,
and write $\SubPrSet{\beta-}$ for $\bigcup_{\beta\leq\gamma<IN}\SubPrSet{\gamma}$.

We write $\HatGame{A}{K}{V}{I}$ 
for the hat game (or simply game) determined
by these four parameters.
If $IN=1$,
we write $\HatGame{A}{K}{V}{1}$.
In the hat game $\mathcal{G} = \HatGame{A}{K}{V}{I}$,
for each $a\in A$ and $f\in \MapSet{A}{K}$,
define $v_a^f = f\restric V(a)$ as 
prisoner $a$'s view, 
that is, 
the part of the coloring $f$ which prisoner $a$ sees.
When $IN=1$,
each prisoner can only base their guess on $v_a^f$,
hence 
a strategy for prisoner $a$ in such a game 
is a function $\map{\strategy{a}}{\MapSet{V(a)}{K}}{K}$.
In the puzzle,
each prisoner declares 
$\sigma_a(f)=\strategy{a}(v_a^f)$.
If two colorings look the same to prisoner $a$,
that is, $f\restric V(a)=g\restric V(a)$,
then $\sigma_a(f)=\sigma_a(g)$.

When $IN\geq 2$, 
the strategy for prisoner $a\in \SubPrSet{m}$ (with $m\geq 1$) 
may also depend on the declarations 
heard from all prisoners in $H(a)$.
If $\{\sigma_b\}_{b\in H(a)}$ are already defined,
and for each $f\in \MapSet{A}{K}$ we define 
$h_a^f = \{\op{b}{\sigma_b(f)} \mid b\in H(a)\}$,
then the strategy $\strategy{a}$ is a function $\map{\strategy{a}}{\MapSet{H(a)}{K}\times \MapSet{V(a)}{K}}{K}$,
and we set $\sigma_a(f)=\strategy{a}(h_a^f,v_a^f)$.
If two colorings $f$ and $g$
share the same view from $a$'s perspective 
($f\restric V(a)=g\restric V(a)$) 
and the same declarations 
from $H(a)$ ($\sigma_b(f)=\sigma_b(g)$ for all $b\in H(a)$),
then $\sigma_a(f)=\sigma_a(g)$.

Prisoner $a$ guesses correctly
in coloring $f$ if $\sigma_a(f)=f(a)$,
and guesses incorrectly otherwise.
A predictor $P$ 
for $\mathcal{G}$ 
is a function $\map{P}{\MapSet{A}{K}}{\MapSet{A}{K}}$ 
such that there exists a set of strategies 
$\{\sigma_a\}_{a\in A}$ producing $P$ by $P(f)=\{\op{a}{\sigma_a(f)} \mid a\in A\}$.
We call $\sigma_a(f)$ the guess that $P(f)$ 
makes for prisoner $a$ under coloring $f$.
If $P(f)(a)=f(a)$, 
then $a$ guesses correctly under $f$.
If $f,g$ are colorings,
we let $\Match{f}{g} = \{a \mid f(a)=g(a)\}$
be the set of prisoners on which $f$ and $g$ match. 
Hence $\Match{f}{P(f)}$ is the set of prisoners 
who guess correctly under $f$ for predictor $P$.
Hence $A\setminus \Match{f}{P(f)}=\{a \mid P(f)(a)\neq f(a)\}$.
For a hat game $\mathcal{G}=\HatGame{A}{K}{V}{I}$,
 we define two families of predictors:
\[
  \PS{\mathcal{G}}{c}{\geq}{n} := \left\{ P \,\middle|\,
    \forall f\in \MapSet{A}{K},\ |\Match{f}{P(f)}|\geq n
  \right\},
\]
\[  \PS{\mathcal{G}}{e}{\leq}{n} 
	:= \left\{ P \,\middle|\,    \forall f\in \MapSet{A}{K},\ 
		|A\setminus \Match{f}{P(f)}|\leq n  \right\}.\]
That is, $\PS{\mathcal{G}}{c}{\geq}{n}$ 
is the set of all predictors guaranteeing 
that at least $n$ prisoners guess correctly
for any assignment $f$,
while $\PS{\mathcal{G}}{e}{\leq}{n}$ is the set guaranteeing 
that at most $n$ prisoners guess incorrectly for any assignment $f$. 

Thus the five parameters 
from the puzzle descriptions correspond formally to:
\begin{itemize}
  \item[(E1)] The cardinality of $A$,
  \item[(E2)] The cardinality of $K$,
  \item[(E3)] The visibility graph,
  \item[(E4)] The inning function,
  \item[(E5)] 
    The required property of a predictor 
  	that ensures the winning condition for prisoners.
\end{itemize}

\section{Previous research on simultaneous puzzles}
\label{section:PreviousResearch}

We now review existing results regarding puzzles 
in which prisoners declare 
simultaneously,
focusing on the case where the number of prisoners is finite.
Since these results have been discussed in prior research,
there is nothing essentially new about the arguments below;
we simply restate them in our own form,
allowing comparison with the non-simultaneous case.

In simultaneous declaration puzzles, 
we often aim to construct $P\in \PS{\mathcal{G}}{c}{\geq}{1}$,
that is, 
a predictor that guarantees at least one prisoner guesses correctly.
For the finite-color case,
if the number of prisoners is not less than the number of colors and every prisoner sees all the other hats (like Puzzle \ref{puz:22VcompleteS}),
then such a predictor exists (Theorem \ref{thm:FFVcompleteS}).
If there are only 2 colors,
even with a sparser visibility graph,
we can still construct such a predictor (Theorem \ref{thm:F2VcyclicS}).
However,
if there are infinitely many colors,
no predictor can guarantee 
at least one correct guess 
in every assignment (Theorem \ref{thm:FIVanyS}).
We begin with a lemma used in the proofs of these statements.

First,
we observe a known result on the average number 
of correct guesses 
across all colorings in a predictor 
for a simultaneous puzzle.
The proof is based on Lemma 2.2.1 of \cite{Hat001}.

\begin{lemma}[]
  \label{lem:average}
  Let $A,K$ be finite sets,
  and consider a game $\mathcal{G}=\HatGame{A}{K}{V}{1}$.
  For any predictor $P$,
  \[    
    \frac{\text{(total number of correct guesses over all colorings)}}{\text{(number of all colorings)}}
    \;=\; \frac{|A|}{|K|}.  \qedhere
  \] 
\end{lemma}

\begin{proof}
  Fix any $a\in A$ and consider 
  all partial colorings $f_a$ on $A\setminus\{a\}$.
  For each such $f_a$,
  there are $|K|$ ways to extend it to a coloring 
  in $\MapSet{A}{K}$ by assigning a color to prisoner $a$.
  Out of these $|K|$ possible extensions,
  exactly one will be correct for prisoner $a$ under $P$,
  namely the one where the color of $a$ 
  is $P(f_a\cup\{\op{a}{k}\})(a)$ for any fixed $k\in K$.
  Hence the number of colorings 
  in which $a$ is correct 
  equals $|\MapSet{A\setminus\{a\}}{K}|=|K|^{|A|-1}$.
  Summing over all $a\in A$,
  the total number of correct guesses 
  across all prisoners and all colorings is
  $|A|\cdot |K|^{|A|-1}$,
  while the total number of colorings is $|K|^{|A|}$.
  Thus
  \[    
    \frac{\text{(total correct)}}{\text{(total colorings)}} 
    \;=\; \frac{|A|\cdot |K|^{|A|-1}}{|K|^{|A|}} 
    \;=\; \frac{|A|}{|K|},  
  \]
  establishing the claimed average.
\end{proof}

It immediately follows that, 
if both $A$ and $K$ are finite and $|A|<|K|$,
no predictor can guarantee 
at least one correct guess in every coloring.

\begin{corollary}[]
  \label{cor:A<K-notMP}
  Let $\mathcal{G}=\HatGame{A}{K}{V}{1}$ be a game where $A$ and $K$ are finite and $|A|<|K|$.
  Then no predictor can guarantee that 
  at least one prisoner guesses correctly.
  In other words,
  $\PS{\mathcal{G}}{c}{\geq}{1} = \emptyset$.
\end{corollary}

\begin{proof}
  Since $|A|<|K|$,
  we have $\tfrac{|A|}{|K|}<1$.
  By Lemma \ref{lem:average},
  every predictor has an average of less than one correct guess per coloring,
  and therefore no predictor can guarantee at least one correct guess in every coloring.
  Hence $\PS{\mathcal{G}}{c}{\geq}{1}=\emptyset$.
\end{proof}

Next,
we present the classical result that if $|A|=|K|$ is finite 
and the visibility graph is complete (each prisoner sees every other hat),
then there is a predictor ensuring at least one correct guess for every coloring.
The following proof is based on \cite{Hat001} (Theorem 2.2.2).

\begin{theorem}[]
  \label{thm:FFVcompleteS}

  In the game $\HatGame{A}{K}{V}{1}$ for a finite set $K$ satisfying $|A| = |K|$,
  the following statements are equivalent.

  \begin{itemize}
    \item[(a)] $V$ is a complete graph.
    \item[(b)] 
      There is a predictor guaranteeing 
      at least one correct guess 
      in every coloring,
      that is,
      $\PS{\mathcal{G}}{c}{\geq}{1}\neq \emptyset$. \qedhere
  \end{itemize}
\end{theorem}

\begin{proof}
  Let $n=|A|=|K|$.
  \begin{itemize}
    \item[] \underline{(a)$\Rightarrow$(b)}
      Label the prisoners $A=\{0,\dots,n-1\}$ 
      and identify $K$ 
      with the additive group 
      $\mathbb{Z}/n\mathbb{Z}=\{\overline{0},\dots,\overline{n-1}\}$.

      First we describe the idea of the strategy.
      Each prisoner $m$ is assigned the value $\overline{m}$.
      Prisoner $m$ bets that the total sum of all hat colors, including their own hat,
      is equal to this assigned value $\overline{m}$.
      Thus, prisoner $m$ guesses their own color by subtracting the sum of the visible hats
      from $\overline{m}$.
      Then, for any coloring, the total sum of all hat colors is some element of $K$,
      and the prisoner assigned to that element guesses correctly.
      Formally, define prisoner $m$'s strategy for each $f_m\in \MapSet{V(m)}{K}$ by
        \[  
          \strategy{m}(f_m) \;=\; \overline{m}-\sum_{b\in \dom(f_m)} f_m(b).     
        \]
      Let $P$ be the predictor generated by these strategies.
      More explicitly, $P$ is generated by the family $\{\sigma_m\}_{m\in A}$,
      where
      $P(f)(m)=\sigma_m(f)=\strategy{m}(v_m^f)$
      for each coloring $f\in \MapSet{A}{K}$ and each $m\in A$.
      In any coloring $f\in \MapSet{A}{K}$,
      if $\displaystyle \sum_{b\in A} f(b)=\overline{m}$, 
      then prisoner $m$ guesses correctly.
      In fact, for prisoner $m$,
      \[
        P(f)(m)
        \;=\;
        \overline{m}-\sum_{b\in \dom(v_m^f)} v_m^f(b).
      \]
      Since the visibility is complete,
      $\dom(v_m^f)=V(m)=A\setminus\{m\}$,
      so
      \begin{align*}
        P(f)(m)
        &\;=\; \hspace{0.5cm}\overline{m}
           &\hspace{-1cm}- &\hspace{0.5cm}\sum_{b\in A\setminus\{m\}} v_m^f(b)\\
        &\;=\; \sum_{b\in A} f(b)
           &\hspace{-1cm}- &\hspace{0.5cm}\sum_{b\in A\setminus\{m\}} f(b)
           \hspace{0.5cm}=\; f(m).
      \end{align*}
      Hence prisoner $m$ 
      guesses correctly 
      and $\PS{\mathcal{G}}{c}{\geq}{1}\neq\emptyset$.

    \item[] \underline{(b)$\Rightarrow$(a)}
      For the contrapositive,
      assume $V$ is not complete,
      that is, 
      there exist distinct $a,b$ with $\op{a}{b}\notin V$.
      Then prisoner $a$ does not see $b$'s hat.
      We show any predictor $P$ fails to guarantee 
      at least one correct guess for all colorings,
      that is, 
      we can find a coloring 
      where everyone guesses incorrectly.

      Let $P$ be arbitrary.
      There must be a coloring $f_a\in \MapSet{A}{K}$ where $a$ is correct:
      set $f_a=\cv{f}{a}{P(f)(a)}$ for some coloring $f\in \MapSet{A}{K}$.
      Now define $f_b=\cv{f_a}{b}{P(f_a)(b)}$.
      Observe that from $a$'s viewpoint,
      $f_b$ and $f_a$ differ only on $b$'s hat,
      which $a$ cannot see;
      thus $\sigma_a(f_b)=\sigma_a(f_a)$,
      so $a$ remains correct. Similarly,
      $b$ is correct in $f_b$ because $P(f_b)(b)=P(f_a)(b)$.
      So in $f_b$,
      at least these two prisoners are correct.
      Meanwhile from $\frac{|A|}{|K|}=1$ and the argument in Lemma \ref{lem:average},
      there must also be a coloring where no one is correct.
      Hence $P$ is not in $\PS{\mathcal{G}}{c}{\geq}{1}$,
      and so no predictor can be in $\PS{\mathcal{G}}{c}{\geq}{1}$. \qedhere
  \end{itemize}
\end{proof}

Next,
again from \cite{Hat001} (Theorem 2.2.2),
we have the well-known 2-color result.
Even if each prisoner does not see all the other hats,
there can still be a predictor guaranteeing at least one correct guess 
in every coloring if the visibility graph is ``cyclic.''
(Puzzle \ref{puz:22VcompleteS} is a special case with a complete graph.)

\begin{theorem}
  \label{thm:F2VcyclicS}
  For a game $\HatGame{A}{K}{V}{1}$ 
  where $A$ is finite and $|K|=2$, 
  the following are equivalent:
  \begin{enumerate}
    \item[(a)] $V$ is cyclic.
    \item[(b)] 
      There is a predictor guaranteeing 
      at least one correct guess 
      in every coloring, 
      that is,
      $\PS{\mathcal{G}}{c}{\geq}{1}\neq \emptyset$. \qedhere
  \end{enumerate}
\end{theorem}

\begin{proof}
  Since $|K|=2$,
  identify $K$ with $\{\overline{0},\overline{1}\}=\mathbb{Z}/2\mathbb{Z}$.
  \begin{itemize}
    \item[] \underline{(a)$\Rightarrow$(b)}
      Let $a_0 \CanSee{V} a_1 \CanSee{V} \dots \CanSee{V} a_{n-1} \CanSee{V} a_n \CanSee{V} a_0$ be a cycle in $V$.
      Then we have $\{a_0, \dots, a_n\} \subseteq A$.

      We first describe the idea of the strategy.
      The prisoner $a_0$ declares the same color as the color they see on $a_1$.
      Each other prisoner on the cycle declares the color different from the color
      they see on the next prisoner in the cycle.
      Then at least one prisoner on the cycle must guess correctly.

      Formally, define strategies for these prisoners as follows:
      For each $f_i\in\MapSet{V(a_i)}{K}(0\leq i\leq n)$, let
      \[
        \begin{aligned}
          \strategy{a_0}(f_0)
            &\;=\; f_0(a_1),\\
          \strategy{a_i}(f_i)
            &\;=\; \overline{1}-f_i(a_{i+1})
            && (1\leq i<n),\\
          \strategy{a_n}(f_n)
            &\;=\; \overline{1}-f_n(a_0).
        \end{aligned}
      \]
      The strategies of prisoners not on the cycle may be arbitrary.
      Let $P$ be the predictor generated by the resulting family of strategies.
      In other words, $P$ is generated by a family $\{\sigma_a\}_{a\in A}$
      satisfying $\{\sigma_{a_i}\}_{0\leq i\leq n}\subseteq \{\sigma_a\}_{a\in A}$.
      Then for any coloring $f\in \MapSet{A}{K}$,
      at least one prisoner in $\{a_0,\dots,a_n\}$ 
      guesses correctly.
      In fact, suppose that all prisoners on the cycle guess incorrectly,
      and assume that $a_0$ wears color $\overline{0}$.
      Since $a_0$ guesses incorrectly, the declaration of $a_0$ is $\overline{1}$.
      Hence $a_1$ wears color $\overline{1}$.
      Next, since $a_1$ also guesses incorrectly, the declaration of $a_1$ is $\overline{0}$.
      Hence $a_2$ wears color $\overline{1}$.
      Continuing in this way, we see that
      the reason why $a_n$ guesses incorrectly is that $a_0$ wears color $\overline{1}$.
      This contradicts the assumption that $a_0$ wears color $\overline{0}$.
      The case where $a_0$ wears color $\overline{1}$ is analogous.
      
    \item[] \underline{(b)$\Rightarrow$(a)}
      We prove the contrapositive.  
      It suffices to show that, 
      for any predictor $P$, there exists a coloring in which all prisoners guess incorrectly.  
      Then $P \notin \PS{\mathcal{G}}{c}{\geq}{1}$, 
      and hence $\PS{\mathcal{G}}{c}{\geq}{1} = \emptyset$.  
      Let $P$ be any predictor for the game $\mathcal{G}$.

      We first describe the idea of constructing a coloring in which every prisoner guesses incorrectly.
      Since the visibility graph is acyclic,
      there are prisoners who see no other prisoners.
      Such prisoners, namely the prisoners in $B_0$ defined below,
      see no hats, and so it is easy to choose their hat colors so that they guess incorrectly.
      Next, there are prisoners who see only prisoners in $B_0$,
      namely the prisoners in $B_1$ defined below.
      Once the hat colors of the prisoners in $B_0$ have been fixed,
      we can also choose the hat colors of the prisoners in $B_1$
      so that they guess incorrectly.
      Continuing in this way, assigning hat colors in order to prisoners who see only prisoners whose hat colors have already been fixed,
      we can construct a coloring in which every prisoner guesses incorrectly.

      We now make this construction formal.
      Define $B_0=\{a\in A \mid V(a)=\emptyset\}$ and, for $m>0$,
      \[
        B_m=\{a\in A \mid V(a)\subseteq \bigcup_{i<m}B_i\}\setminus\bigcup_{i<m}B_i.
      \]
      Since $V$ is acyclic, $B_0\neq\emptyset$.
      By construction, the sets $B_i$ are pairwise disjoint.
      Moreover, since $A$ is finite, there exists $n\in\mathbb{N}$ such that
      \[
        A=\bigcup_{0\le i\le n}B_i,
        \qquad
        B_i\neq\emptyset \text{ for all } i\le n.
      \]

      Each prisoner in $B_0$ sees no hat, 
      so $\strategy{a}$ is a constant function for every $a\in B_0$.  
      For any $f\in\MapSet{A}{K}$, define
      \[
        f_0
        \;=\;
        \bigl\{
          \op{a}{\overline{1}-P(f)(a)}
          \,\bigm|\,
          a\in B_0
        \bigr\}.
      \]
      In any extension of $f_0$, all prisoners in $B_0$ guess incorrectly.  
      Inductively, for $m>0$, define
      \[
        f_m
        \;=\;
        f_{m-1}
        \;\cup\;
        \bigl\{
          \op{a}{\overline{1}-\strategy{a}(f_{m-1}\restric V(a))}
          \,\bigm|\,
          a\in B_m
        \bigr\}.
      \]
      Then, for every $m\le n$, all prisoners in $\bigcup_{i\le m}B_i$ guess incorrectly under $P$.  
      Thus the final coloring $f_n$ makes every prisoner in $A$ guess incorrectly, as desired.  
      \qedhere

  \end{itemize}
\end{proof}

Finally,
Proposition 3 of \cite{Hat085}
shows that 
if $A$ is finite and $K$ is infinite,
no predictor can guarantee at least one prisoner 
guesses correctly:
\begin{theorem}
  \label{thm:FIVanyS}
  For $\HatGame{A}{K}{V}{1}$ with finite $A$ and infinite $K$,
  no predictor achieves 
  at least one correct guess on every coloring.
  In other words,
  $\PS{\mathcal{G}}{c}{\geq}{1}=\emptyset$.
\end{theorem}

\begin{proof}
  Suppose $\PS{\mathcal{G}}{c}{\geq}{1}\neq\emptyset$,
  so some predictor $P$ guarantees at least one correct guess under every coloring.
  Since $A$ is finite,
  let $n-1=|A|$.
  Pick any $n$-element subset $K_n\subseteq K$.
  Consider $C_n=\{\map{f}{A}{K_n}\mid \ran(f)\subseteq K_n\}$,
  i.e.\ all colorings from $A$ into these $n$ chosen colors.
  Fix an arbitrary $k\in K_n$ and define $P_n$ as follows:
  \[
    P_n(f)(a) = 
      \begin{cases}
        P(f)(a) & \text{if }P(f)(a)\in K_n,\\
        k & \text{otherwise}.
      \end{cases}
  \]
  Then $P_n$ is a predictor in the game $\mathcal{G}_n=\HatGame{A}{K_n}{V}{1}$,
  and it still guarantees at least one correct guess under each $f\in C_n$.
  By Lemma \ref{lem:average},
  every predictor in $\mathcal{G}_n$ must have 
  an average of $|A|/|K_n|=(n-1)/n<1$ correct guesses.
  This contradicts $P_n$ guaranteeing 
  at least one prisoner correct in every coloring.
\end{proof}

Hence, in simultaneous puzzles with a finite number of prisoners,
the existence of a predictor that guarantees at least one correct guess
for every coloring depends on the number of colors.
Equivalently, this amounts to determining whether
$\PS{\mathcal{G}}{c}{\geq}{1}$ is nonempty:
\begin{itemize}
  \item For 2 colors,
  $\PS{\mathcal{G}}{c}{\geq}{1}\neq\emptyset$
  if and only if the visibility graph is cyclic.

  \item For 3 or more finite colors with $|A|=|K|$,
  $\PS{\mathcal{G}}{c}{\geq}{1}\neq\emptyset$
  if and only if the visibility graph is complete.

  \item For infinitely many colors,
  $\PS{\mathcal{G}}{c}{\geq}{1}=\emptyset$.
\end{itemize}
In the infinite prisoner case,
it is possible to guarantee 
not just one but infinitely many correct guesses.
In particular,
under the Axiom of Choice,
there is a predictor under which only finitely many prisoners guess incorrectly,
known as 	a ``finite error'' predictor.
We introduce one further notion.

\begin{definition}[]
  For a game $\mathcal{G}=\HatGame{A}{K}{V}{1}$,
  a predictor $P$ is called a finite error predictor 
  if every coloring has only finitely many prisoners 
  who guess incorrectly,
  that is, $P(f)=^{*}f$ for all colorings $f$.
  Here $=^{*}$ means they differ only on finitely many coordinates.
  Denote the set of such predictors by $\FEPS{\mathcal{G}}$.
\end{definition}

According to \cite{Hat004},
the existence of a finite error predictor 
for an infinite number of prisoners 
(with complete visibility) 
was discovered in 2004 by two graduate students.
It is often cited
as the starting point of infinite hat puzzle research:

\begin{theorem}[Gabay--O'Connor]
  \label{thm:FEP}
  Consider the game $\mathcal{G}=\HatGame{A}{K}{V}{1}$
  where $A$ is infinite and $V$ is complete.
  Under the Axiom of Choice,
  there is a finite error predictor.
  That is,
  $\FEPS{\mathcal{G}}\neq\emptyset$.
\end{theorem}

\begin{proof}
  Consider the equivalence relation $=^{*}$ on $\MapSet{A}{K}$ 
  given by $f=^{*}g$ if and only if they differ only on finitely many prisoners.
  Let $[f]$ be the equivalence class of $f$ and choose a selector $\varphi$ on these classes using the Axiom of Choice,
  that is, $\varphi([f])\in [f]$.
  In particular, $f=^{*}\varphi([f])$.
  Also, if $f=^{*}g$, then $[f]=[g]$, and hence $\varphi([f])=\varphi([g])$.

  We first describe the idea of the strategy.
  The prisoners share one such selector $\varphi$ in advance.
  From the coloring visible to a prisoner,
  the prisoner can determine the equivalence class to which the actual coloring belongs:
  changing only their own hat color does not change the equivalence class.
  Then, by applying the shared selector $\varphi$ to that equivalence class,
  the prisoner can identify the representative coloring of the class.
  The prisoner then declares the color they have in that representative coloring.
  The following formal construction shows that such strategies can be defined
  and that all but finitely many prisoners guess correctly.

  Fix an element $k\in K$.
  For each prisoner $a\in A$, define the strategy $\strategy{a}$ by
  \[
    \strategy{a}(f_a)=\varphi\bigl(\; [f_a\cup\{\op{a}{k}\}] \;\bigr)(a)
  \]
  for each $f_a\in \MapSet{V(a)}{K}$.
  Let $P$ be the predictor generated by the family $\{\sigma_a\}_{a\in A}$.
  We show that $P(f)=^{*}f$ for every $f\in \MapSet{A}{K}$.

  Fix $f\in \MapSet{A}{K}$.
  Since $V$ is complete, each prisoner $a\in A$ sees every hat except their own.
  Hence $V(a)=A\setminus\{a\}$ and
  $v_a^f=f\restric V(a)=f\restric(A\setminus\{a\})$.
  Therefore, by the definition of $\strategy{a}$,
  \begin{equation*}
    \sigma_a(f)
    \;=\; \strategy{a}(v_a^f)
    \;=\; \varphi\bigl(\; [\cv{f}{a}{k}] \;\bigr)(a).
  \end{equation*}
  Since $f=^{*}\cv{f}{a}{k}$, we have
  $\varphi([f])=\varphi\bigl(\; [\cv{f}{a}{k}] \;\bigr)$.
  Put $f_{\varphi}=\varphi([f])$.
  Then $\sigma_a(f)=f_{\varphi}(a)$ for every $a\in A$.
  Hence
  \[
    P(f)
    \;=\;
    \bigl\{\op{a}{\sigma_a(f)} \,\bigm|\, a\in A\bigr\}
    \;=\;
    \bigl\{\op{a}{f_{\varphi}(a)} \,\bigm|\, a\in A\bigr\}
    \;=\;
    f_{\varphi}.
  \]
  By the choice of $\varphi$, we have $f=^{*}f_{\varphi}$.
  Therefore $P(f)=^{*}f$.
  Thus $P$ is a finite error predictor, and so $\FEPS{\mathcal{G}}\neq\emptyset$.
\end{proof}

In non-simultaneous puzzles with infinitely many prisoners,
a similar result holds 
under the Axiom of Choice,
namely the existence of a predictor with at most one error 
if conditions allow for a ``hint''.
We will see this in Corollary \ref{cor:IAVA}.

\section{Non-simultaneous puzzle with finite prisoners}
\label{section:FinitePrisoner}

Throughout this section, assume $A$ is finite.

For a game $\HatGame{A}{K}{V}{I}$ with $IN\geq 2$,
define the following three conditions:

\begin{itemize}
  \item[(S1)]
    Exactly one prisoner declares in the initial inning,
    that is, there is a unique $s\in A$ 
    with $H(s)=\emptyset$, or $|\SubPrSet{0}|=1$.
  \item[(S2)]
    Each prisoner either sees or hears 
    the hat (or declaration) 
    of every other prisoner,
    that is, 
    $\forall a\in A(\; V(a)\cup H(a)=A\setminus\{a\} \;)$.
  \item[(S3)]
    No prisoner both sees and hears the same prisoner,
    that is, 
    $\forall a\in A(\; V(a)\cap H(a)=\emptyset \;)$.
\end{itemize}
Many well-known hat puzzles satisfy some of these conditions.
Puzzle \ref{puz:52VcompleteA} satisfies (S1) and (S2).
Puzzle \ref{puz:52VstairsA} 
satisfies all three (S1), (S2), and (S3).
In games such as Puzzle \ref{puz:52VcompleteA} and Puzzle \ref{puz:52VstairsA},
which satisfy (S1) and aim to guarantee at most one incorrect guess,
the main role of the first speaker is not to guess correctly themselves,
but to provide a hint to the prisoners who speak later.
The first prisoner encodes the configuration of the other prisoners' hats as a single color,
and the later prisoners use that information to infer their own hat colors.

To encode the configuration of the other hats as a single color in this way,
one needs a mechanism for combining colors and for recovering a particular hat color from the encoded information.
Thus, we first consider the case in which the color set $K$ carries an abelian group structure.
With such a structure, the first prisoner can declare the sum of the other prisoners' hat colors as a hint,
and the later prisoners can recover their own hat colors by subtracting the colors already known to them from that sum.
The next lemma shows that, for finitely many prisoners,
this basic method applies to any abelian group $K$, regardless of its cardinality,
and it serves as a fundamental general result of this section.
In particular, it applies to Puzzle \ref{puz:52VcompleteA} and Puzzle \ref{puz:52VstairsA}.

On the other hand, as in simultaneous puzzles,
the conditions for the existence of a predictor in a non-simultaneous puzzle depend on the number of colors.
After Lemma \ref{lem:FGVA},
we consider separately the cases of two colors, three or more finite colors, and infinitely many colors,
and investigate how the conditions on the prisoners' visibility and on their ability to hear other prisoners' declarations differ from those in Lemma \ref{lem:FGVA} when one requires a predictor with at most one error.

Previous work on simultaneous puzzles with finitely many prisoners established not only conditions guaranteeing the existence of a predictor,
but also conditions equivalent to its existence.
In the latter part of this section,
we similarly investigate equivalences between the existence of a predictor and conditions on finite non-simultaneous games.

Finally, as after Theorem \ref{thm:FIVanyS} in Section \ref{section:PreviousResearch},
where simultaneous puzzles were discussed,
we summarize these results according to the number of colors.

\begin{lemma}[]
  \label{lem:FGVA}
  Let $(K,+)$ be an abelian group.
  Suppose $IN\geq 2$ and the game $\mathcal{G}=\HatGame{A}{K}{V}{I}$ 
  satisfies (S1) and (S2).
  Then there is a predictor in $\PS{\mathcal{G}}{e}{\leq}{1}$;
  in other words,
  one can guarantee at most one prisoner guesses incorrectly.
\end{lemma}

\begin{proof}
  Since (S1) holds, write $\SubPrSet{0}=\{s\}$.
  The strategy of prisoner $s$ is to sum all the colors they see and declare that sum.
  Formally, for a coloring $f_s\in \MapSet{V(s)}{K}$, define
  \[
    \strategy{s}(f_s)
    \;=\;
    \sum_{b\in \dom(f_s)} f_s(b).
  \]
  The remaining prisoners subtract two quantities from prisoner $s$'s declaration:
  the sum of the declarations, except for $s$'s declaration,
  that they have heard before their own turn,
  and the sum of all visible hat colors of prisoners who have not yet declared,
  including those who declare at the same time as themselves.
  Formally, for $a\in \SubPrSet{m}$ with $m\ge 1$,
  and for two colorings $f_a\in \MapSet{V(a)}{K},h_a\in \MapSet{H(a)}{K}$,
  define
  \[
    \strategy{a}(h_a,f_a)
    \;=\;
    h_a(s)
    \;-\;
    \sum_{b\in \dom(h_a)\setminus\{s\}} h_a(b)
    \;-\;
    \sum_{b\in \dom(f_a)\setminus \dom(h_a)} f_a(b).
  \]
  Let $P$ be the predictor generated by the family
  $\{\sigma_s\}\cup\{\sigma_a\}_{a\in \SubPrSet{1-}}$.
  Under this predictor $P$, in any coloring $f\in \MapSet{A}{K}$,
  prisoner $s$ may err,
  but all other prisoners guess correctly.
  We now prove this formally.

  Fix any $f\in \MapSet{A}{K}$.
  We show that every $a\in \SubPrSet{1-}$ guesses correctly.
  We prove this by induction on $n=1,\dots,IN-1$.
  We first consider the prisoners in $\SubPrSet{1}$.
  Take $a\in\SubPrSet{1}$.
  Since $\SubPrSet{0}=\{s\}$, prisoner $a$ has heard only $s$'s declaration before their own turn.
  Hence $\dom(h_a^f)\setminus\{s\}=\varnothing$.
  Also, since (S2) implies $V(s)=A\setminus\{s\}$,
  \[
    h_a^f(s)
    \;=\;
    \sigma_s(f)
    \;=\;
    \sum_{b\in \SubPrSet{1-}} f(b).
  \]
  Moreover, by (S2),
  \[
    \dom(v_a^f)\setminus \dom(h_a^f)
    \;=\;
    \SubPrSet{1-}\setminus\{a\}.
  \]
  Therefore
  \begin{align*}
    P(f)(a)
    \;&=\;
    \strategy{a}(h_a^f,v_a^f) \\
    &=\; h_a^f(s)
       \;-\;
       \sum_{b\in \dom(v_a^f)\setminus \dom(h_a^f)} v_a^f(b) \\
    &=\;
       \sum_{b\in \SubPrSet{1-}} f(b)
       \;-\;
       \sum_{b\in \SubPrSet{1-}\setminus\{a\}} f(b)
     \;=\; f(a).
  \end{align*}
  Thus every prisoner in $\SubPrSet{1}$ guesses correctly.

  Now let $n\in\{2,\dots,IN-1\}$,
  and suppose that all prisoners in $\SubPrSet{1},\dots,\SubPrSet{n-1}$ guess correctly.
  Take $a\in \SubPrSet{n}$.
  Then
  \begin{align*}
    P(f)(a)
    \;&=\;
    \sigma_a(f)
    \;=\;
    \strategy{a}(h_a^f,v_a^f) \\
    &=\; h_a^f(s)
       \;-\;
       \sum_{b\in \dom(h_a^f)\setminus\{s\}} h_a^f(b)
       \;-\;
       \sum_{b\in \dom(v_a^f)\setminus \dom(h_a^f)} v_a^f(b).
       \tag{$\dagger$}
  \end{align*}
  Note
  $h_a^f(s)=\sigma_s(f)=\sum_{b\in V(s)}f(b)$
  and since (S2) implies $V(s)=A\setminus\{s\}$,
  \[
    \sigma_s(f)
    \;=\;
    \sum_{b\in A\setminus\{s\}} f(b)
    \;=\;
    \sum_{b\in \SubPrSet{1-}} f(b).
    \tag{$\dagger1$}
  \]
  By the induction hypothesis,
  all prisoners in $\SubPrSet{1}\cup\dots\cup\SubPrSet{n-1}$ guess correctly, so
  \[
    \sum_{b\in \dom(h_a^f)\setminus\{s\}} h_a^f(b)
    \;=\;
    \sum_{b\in \SubPrSet{1-(n-1)}} f(b).
    \tag{$\dagger2$}
  \]
  Also from (S2),
  we get
  \[
    \dom(v_a^f)\setminus \dom(h_a^f)
    \;=\;
    \SubPrSet{n-}\setminus\{a\}.
    \tag{$\dagger3$}
  \]
  Thus plugging $(\dagger1)$--$(\dagger3)$ into $(\dagger)$,
  we obtain
  \[
    P(f)(a)
    \;=\;
    \sum_{b\in \SubPrSet{1-}} f(b)
    \;-\;
    \sum_{b\in \SubPrSet{1-(n-1)}} f(b)
    \;-\;
    \sum_{b\in \SubPrSet{n-}\setminus\{a\}} f(b)
    \;=\;
    f(a).
  \]
  Hence prisoner $a$ guesses correctly.
  Thus $P\in \PS{\mathcal{G}}{e}{\leq}{1}$.
\end{proof}

For the case of three or more finite colors,
the same approach as in Lemma \ref{lem:FGVA} works if (S1) and (S2) hold,
because one can form 
an abelian group operation in $K$ 
(for example, $\mathbb{Z}/|K|\mathbb{Z}$ if $K$ is finite).
Thus:

\begin{theorem}[]
  \label{thm:FFVA}
  Let $K$ be a finite set with $3\le|K|<\omega$ 
  and consider a game $\HatGame{A}{K}{V}{I}$ with $IN\ge2$.
  If (S1) and (S2) hold,
  then $\PS{\mathcal{G}}{e}{\leq}{1}\neq\emptyset$.
\end{theorem}

Looking at the proof of Lemma \ref{lem:FGVA}, 
it may appear that, 
in order to construct a predictor with at most one incorrect guess, 
there must be exactly one prisoner who speaks first — that is, 
one who plays the role of giving a hint and may guess incorrectly.  
However, 
when there are only two colors, 
it can be shown that two prisoners may instead speak in the initial inning.
In that case, 
we need to modify the visibility of the prisoners slightly.

\begin{theorem}[]
  \label{thm:F2VA}
  Suppose $|K|=2$, $IN\geq 2$, 
  and the game $\mathcal{G}=\HatGame{A}{K}{V}{I}$ satisfies:
  \begin{itemize}
    \item[(S4)] Every prisoner in $\SubPrSet{0}$ 
    	sees the hats of all other prisoners 
		(so $\forall a\in \SubPrSet{0}(\; V(a)=A\setminus\{a\} \;)$).
    \item[(S5)] 
      Exactly two prisoners declare in the initial inning,
      that is, $|\SubPrSet{0}|=2$.
    \item[(S6)] Every prisoner who declares 
    	later sees the hats of all prisoners 
		in $\SubPrSet{0}$ and of those 
		who declare in the same inning or later, excluding themselves
		($\forall n\in\{1,\dots,IN-1\}\,
			\forall a\in\SubPrSet{n}\,
			\bigl(\; \bigl(\SubPrSet{0}\cup\SubPrSet{n-}\bigr)\setminus\{a\}\subseteq V(a)\;\bigr)$).
  \end{itemize}
  Then there is a predictor in $\PS{\mathcal{G}}{e}{\leq}{1}$.
  In other words, 
  the prisoners can guarantee that at most one guess is incorrect.
\end{theorem}

\begin{proof}
  Before giving the concrete construction of a predictor,
  we first explain the idea of the proof.
  Identify $K$ with
  $\mathbb{Z}/2\mathbb{Z}=\{\overline{0},\overline{1}\}$.
  By (S5), let
  $\SubPrSet{0}=\{s_0,s_1\}$.

  In order to have at most one incorrect guess in $A$ as a whole,
  there must also be at most one incorrect guess in $\SubPrSet{0}$.
  Since $|K|=2$ and, by (S4), $s_0$ and $s_1$ can see each other's hats,
  the situation restricted to $\SubPrSet{0}$ is the same as in
  Puzzle \ref{puz:22VcompleteS}.
  We therefore use the same idea as in the proof of
  Theorem \ref{thm:FFVcompleteS}, which generalizes the strategy in
  Puzzle \ref{puz:22VcompleteS}.
  Namely, assign $\overline{0}$ to $s_0$ and $\overline{1}$ to $s_1$.
  Each $s_i$ can declare so as to guess correctly when the total sum of
  the hat colors of all prisoners in $A$ is equal to the assigned value
  $\overline{i}$.
  Thus exactly one of $s_0$ and $s_1$ guesses correctly
  (this will be proved in Claim 1),
  or equivalently, there is only one incorrect guess in $\SubPrSet{0}$.

  In the present game, however, it is not enough to control the guesses
  in $\SubPrSet{0}$; every prisoner in $\SubPrSet{1-}$ must also guess correctly.
  Unlike the case of Puzzle \ref{puz:52VstairsA}, which satisfies (S3),
  meaning that a prisoner cannot see the hat of a prisoner whose declaration
  they hear, in the present setting (S6) ensures that the prisoners in
  $\SubPrSet{1-}$ not only hear the declarations of $s_0$ and $s_1$ but also
  see their hats.
  Hence, by comparing the actual hat colors with the declarations,
  they can determine which of $s_0$ and $s_1$ guessed correctly.

  We therefore let $s_0$ and $s_1$ follow the above idea in such a way that
  the information of which one guessed correctly conveys the total sum of
  all prisoners' hat colors to the later prisoners in $\SubPrSet{1-}$.
  More precisely, the declarations can be chosen so that, whenever $s_i$
  guesses correctly, the total sum of all prisoners' hat colors is
  $\overline{i}$ (this will be proved in Claim 2).
  Consequently, before making their own declarations, the prisoners in
  $\SubPrSet{1-}$ can determine which of $s_0$ and $s_1$ guessed correctly
  and thereby learn the total sum of the hat colors of all prisoners,
  including themselves.
  From this point on, the argument proceeds inductively as in
  Lemma \ref{lem:FGVA}: each prisoner can recover their own hat color by
  subtracting from this total sum the declarations already heard and the
  visible hat colors.

  We now construct a predictor for which Claims 1 and 2 hold.

  For colorings $f_0\in \MapSet{V(s_0)}{K}$ and
  $f_1\in \MapSet{V(s_1)}{K}$, define
  \[
    \strategy{s_0}(f_0)
    \;=\;
    \sum_{a\in \dom(f_0)\setminus\{s_1\}} f_0(a)
    \; + \; \overline{0}
    \; - \; f_0(s_1),
  \]
  and
  \[
    \strategy{s_1}(f_1)
    \;=\;
    \sum_{a\in \dom(f_1)\setminus\{s_0\}} f_1(a)
    \; + \; \overline{1}
    \; - \; f_1(s_0).
  \]

  Also, for $n\geq1$, $a\in \SubPrSet{n}$,
  $f_a\in \MapSet{V(a)}{K}$, and
  $h_a\in \MapSet{H(a)}{K}$, define
  \[
    \strategy{a}(h_a,f_a)
    \;=\;
    \sum_{\substack{s_i\in \SubPrSet{0}\\ h_a(s_i)=f_a(s_i)}} \overline{i}
    \;-\;
    \sum_{b\in \dom(h_a)\setminus\SubPrSet{0}} h_a(b)
    \;-\;
    \sum_{b\in(\dom(f_a)\setminus\dom(h_a))\cup\SubPrSet{0}} f_a(b).
  \]
  Let $P$ be the predictor generated by
  $\{\sigma_{s_0},\sigma_{s_1}\}
  \cup\{\sigma_a\}_{a\in \SubPrSet{1-}}$.

  Fix any coloring $f\in\MapSet{A}{K}$.
  If Claims 1 and 2 hold for this predictor $P$, then, by the idea described
  above, every prisoner in $\SubPrSet{1-}$ guesses correctly, while exactly
  one prisoner in $\SubPrSet{0}$ guesses incorrectly.
  Thus there is at most one incorrect guess in total, and hence
  $P\in\PS{\mathcal{G}}{e}{\leq}{1}$.
  It therefore remains to prove Claims 1 and 2.

  \begin{itemize}
    \item[] \underline{(Proof of Claim 1)}

      Claim 1 states that exactly one prisoner in $\SubPrSet{0}$ guesses correctly.
      By (S4), each of $s_0$ and $s_1$ can see the hats of all prisoners
      other than themselves. Hence
      \begin{align*}
        P(f)(s_0)
          \;&=\;
          \sum_{a\in\SubPrSet{1-}} f(a)
          \;+\; \overline{0}
          \;-\; f(s_1), \\
        P(f)(s_1)
          \;&=\;
          \sum_{a\in\SubPrSet{1-}} f(a)
          \;+\; \overline{1}
          \;-\; f(s_0).
      \end{align*}

      First, $s_0$ guesses correctly if and only if
      $\sum_{a\in\SubPrSet{1-}} f(a)
      + \overline{0} - f(s_1) = f(s_0)$,
      which is equivalent to $\sum_{a\in A} f(a) = \overline{0}$.
      Similarly, $s_1$ guesses correctly if and only if
      $\sum_{a\in A} f(a) = \overline{1}$.
      Since $K=\{\overline{0},\overline{1}\}$,
      exactly one of these two equalities holds.
      Hence exactly one prisoner in $\SubPrSet{0}$ guesses correctly.

    \item[] \underline{(Proof of Claim 2)}

      Claim 2 states that if $s_i\in\SubPrSet{0}$ guesses correctly, then the
      total sum of all prisoners' hat colors is equal to the value assigned to
      $s_i$; that is, $\sum_{a\in A}f(a)=\overline{i}$.
      Suppose first that $s_0$ guesses correctly.
      Since $P(f)(s_0)=f(s_0)$, we have
      \[
        \sum_{a\in\SubPrSet{1-}} f(a)
        \;+
        \overline{0}
        \;-
        f(s_1)
        \;=
        f(s_0).
      \]
      Moreover,
      $\sum_{a\in\SubPrSet{1-}} f(a)
      = \sum_{a\in A} f(a)-f(s_0)-f(s_1)$,
      so
      \begin{align*}
        \sum_{a\in A} f(a)
        \;- f(s_0)
        \;- f(s_1)
        \;+ \overline{0}
        \;- f(s_1)
          &\;=\; f(s_0),\\
        \sum_{a\in A} f(a)
        \;+ \overline{0}
          &\;=\; 2f(s_0)+2f(s_1).
      \end{align*}
      Since $K=\mathbb{Z}/2\mathbb{Z}$,
      $2f(s_0)=2f(s_1)=\overline{0}$.
      Therefore $\sum_{a\in A} f(a)=\overline{0}$.

      Similarly, if $s_1$ guesses correctly, then $P(f)(s_1)=f(s_1)$,
      and we obtain $\sum_{a\in A} f(a)=\overline{1}$.
      Thus, if $s_i$ guesses correctly, then
      $\sum_{a\in A}f(a)=\overline{i}$.
      \qedhere
  \end{itemize}
\end{proof}

In the case of infinitely many colors,
the assumptions of Lemma \ref{lem:FGVA} can be weakened further
under an additional assumption on the set of colors.
Since the first prisoner's declaration can encode enough information
for all later prisoners to guess correctly,
neither an abelian group structure on $K$ nor the full strength of (S2) is needed.

\begin{theorem}[]
  \label{thm:FIVA}
  Suppose $K$ is infinite,
  $IN\ge2$,
  and $\mathcal{G}=\HatGame{A}{K}{V}{I}$ satisfies (S1) and (S4).
  Assume further that there exists a bijection $\varphi:K\to\MapSet{A\setminus\{s\}}{K}$,
  where $\SubPrSet{0}=\{s\}$.
  Then $\PS{\mathcal{G}}{e}{\leq}{1}\neq\emptyset$.
\end{theorem}

\begin{proof}
  Fix such a bijection $\varphi$.

  The idea of the strategy is as follows.
  The first prisoner $s$ sees the whole coloring on $A\setminus\{s\}$.
  Since every coloring on $A\setminus\{s\}$ can be encoded by a single color of $K$,
  prisoner $s$ declares the color that encodes the coloring they see.
  Thus, for a coloring $f_s\in \MapSet{V(s)}{K}$, define
  \[
    \strategy{s}(f_s)
    \;=\;
    \varphi^{-1}(f_s).
  \]

  The later prisoners use the declaration of $s$ to recover the coloring seen by $s$.
  Since this recovered coloring is a coloring of $A\setminus\{s\}$,
  it contains in particular the hat color of each later prisoner.
  Hence each later prisoner simply decodes $s$'s declaration and declares the color assigned to themselves
  by the decoded coloring.
  Formally, for any $a\in \SubPrSet{1-}$,
  and for two colorings $f_a\in \MapSet{V(a)}{K},h_a\in \MapSet{H(a)}{K}$,
  define
  \[
    \strategy{a}(h_a,f_a)
    \;=\;
    \varphi\bigl(\; h_a(s) \;\bigr)(a).
  \]
  Then for every coloring $f\in \MapSet{A}{K}$,
  prisoner $s$ might fail, but everyone else is correct.
  This yields $\PS{\mathcal{G}}{e}{\leq}{1}\neq\emptyset$.
\end{proof}

The bijection assumed in Theorem \ref{thm:FIVA} exists, for example,
when $K$ is the set of natural numbers or the set of real numbers,
without assuming the Axiom of Choice.
However, its existence is not guaranteed in ZF for an arbitrary infinite set $K$.
In fact, Tarski showed that the assertion that every infinite set is in bijection
with its Cartesian square is equivalent to the Axiom of Choice
(statement (II) in \cite{Set073}).

Under the Axiom of Choice, the bijection required in Theorem \ref{thm:FIVA}
exists for every infinite set $K$.
The following corollary is immediate.

\begin{corollary}
  \label{cor:FIVA-AC}
  Assume the Axiom of Choice.
  Suppose $K$ is infinite, $IN\ge2$,
  and $\mathcal{G}=\HatGame{A}{K}{V}{I}$ satisfies (S1) and (S4).
  Then $\PS{\mathcal{G}}{e}{\leq}{1}\neq\emptyset$.
\end{corollary}

As explained in the introduction to this section,
in the simultaneous case there were theorems asserting that the existence of such a predictor
is equivalent to certain conditions on $A$ and $V$
(for example, Theorems \ref{thm:FFVcompleteS} and \ref{thm:F2VcyclicS}).
We now establish several analogous results for non-simultaneous puzzles.
To obtain the reverse implications ``(b)$\Rightarrow$(a)'' of the theorems proved above,
the game assumption (S3), which has not played a role in the discussion so far, becomes important.
Under the additional assumption (S3),
Theorem \ref{thm:after-thm:FFVA} gives the reverse implication for Theorem \ref{thm:FFVA},
and Theorem \ref{thm:after-thm:FIVA} gives the reverse implication for Theorem \ref{thm:FIVA}.
This latter implication does not require the bijection assumed in Theorem \ref{thm:FIVA}.

Before these two theorems,
we prove two propositions used in their proofs.
These propositions investigate properties that must hold in the game when a non-simultaneous game admits a predictor with at most one error.

The first proposition states that if a non-simultaneous game admits a predictor with at most one error,
then the simultaneous game obtained by restricting to the prisoners who declare in the initial inning
also admits a predictor with the same error guarantee.

\begin{proposition}[]
  \label{prop:FirstGroup}
  Let $\mathcal{G}=\HatGame{A}{K}{V}{I}$ with $IN\ge2$, 
  and assume $\PS{\mathcal{G}}{e}{\leq}{1}\neq\emptyset$.
  Then in the simultaneous game 
  $\mathcal{G}_1=\HatGame{\SubPrSet{0}}{K}{ V\cap (\SubPrSet{0}\times\SubPrSet{0}) }{1}$ 
  (keeping the same set of colors and same visibility relation 
  restricted to $\SubPrSet{0}$),
  there is also a predictor in $\PS{\mathcal{G}_1}{e}{\leq}{1}$.
\end{proposition}

\begin{proof}
  Let $P\in\PS{\mathcal{G}}{e}{\leq}{1}$ and fix $k\in K$.
  For each coloring $f\in\MapSet{\SubPrSet{0}}{K}$,
  let $f_k$ be its extension to $A$ obtained by assigning color $k$
  to every prisoner outside $\SubPrSet{0}$.
  Define the strategy $\strategy{a}$ of each $a\in\SubPrSet{0}$ in $\mathcal{G}_1$ by
  \[
    \strategy{a}(\ f\restric (V(a)\cap\SubPrSet{0})\ )=P(f_k)(a).
  \]
  This is well-defined, since the declaration of each initial prisoner under $P$
  depends only on their visible hat colors, and the colors outside $\SubPrSet{0}$ are fixed.
  Let $P_1$ be the predictor generated by $\{\sigma_a\}_{a\in\SubPrSet{0}}$.
  Since $P$ makes at most one error for $f_k$,
  the predictor $P_1$ makes at most one error for $f$.
  Hence $P_1\in\PS{\mathcal{G}_1}{e}{\leq}{1}$.
\end{proof}

Next,
we gather several useful statements under (S1) and (S3) 
once we know there is a predictor guaranteeing at most one error.

\begin{proposition}[]
  \label{prop:useful}
  Suppose that a game $\mathcal{G}=\HatGame{A}{K}{V}{I}$ with $IN\geq 2$
  satisfies (S1), (S3), and $P\in \PS{\mathcal{G}}{e}{\leq}{1}$.
  Then, for every $f\in \MapSet{A}{K}$, the following hold.
  \begin{enumerate}
    \item
      \label{sub:1}
      (S4).
      That is, the first prisoner $s$ sees the hats of all prisoners other than themselves.
      Hence $V(s)=A\setminus \{ s\}$.
    \item
      \label{sub:2}
      If a prisoner guesses incorrectly under $P$,
      then that prisoner must be the first speaker $s$.
      That is, if $P(f)(a)\neq f(a)$, then $a=s$.
      In other words, no prisoner other than $s$ ever guesses incorrectly.
    \item
      \label{sub:3}
      If exactly one hat color visible to $s$ is changed,
      then $s$'s declaration also changes.
      That is, for any $a\in \SubPrSet{1-}$ and any $k\in K\setminus\{f(a)\}$,
      we have $P(f)(s)\neq P(\cv{f}{a}{k})(s)$.
    \item
      \label{sub:4}
      Furthermore, when the set of colors is finite,
      if $s$'s declaration changes,
      then the declarations made in all later innings also change.
      That is, if $K$ is finite, then for any $a\in\SubPrSet{1-}$
      and any $k\in K\setminus\{h_a^f(s)\}$, we have
      $\strategy{a}(h_a^f,v_a^f)\neq\strategy{a}(\cv{h_a^f}{s}{k},v_a^f)$. \qedhere
  \end{enumerate}
\end{proposition}
\begin{proof}
  \begin{itemize}
    \item[(\ref{sub:1})] 

      Since $P\in \PS{\mathcal{G}}{e}{\leq}{1}$,
      there is at most one incorrect prisoner for $f$.
      Suppose that $V(s)\neq A\setminus \{ s\}$,
      that is, there exists a prisoner $a$ whose hat is not visible to $s$.
      By (S1), $a\in \SubPrSet{1-}$.
      Starting from $f$,
      we construct a coloring $f_s$ in which $s$ guesses incorrectly,
      and then, starting from $f_s$,
      construct another coloring $f_a$ in which $a$ also guesses incorrectly
      without changing the fact that $s$ guesses incorrectly;
      in other words, both $s$ and $a$ guess incorrectly under $P$ for $f_a$. \\
      More precisely, let
      $f_s = \cv{f}{s}{ \nee{P(f)(s)} }$.
      Since $s$ does not see their own hat,
      changing the color of their own hat does not change $s$'s declaration,
      which remains $P(f)(s)$.
      On the other hand, since $f_s(s)\neq P(f)(s)$,
      $s$ guesses incorrectly under $P$ for $f_s$.
      By (S3), no prisoner in a later inning sees $s$'s hat.
      Since $s$'s declaration is unchanged, a simple induction on the innings
      shows that all later declarations are also unchanged.
      In particular, $a$ makes the same declaration for $f$ and $f_s$.
      Moreover, since $s$ guesses incorrectly under $f_s$ and
      $P\in\PS{\mathcal{G}}{e}{\leq}{1}$,
      every prisoner other than $s$, including $a$,
      guesses correctly under $f_s$. \\
      Next, let
      $f_a = \cv{f_s}{a}{\nee{P(f)(a)}}$.
      Since $s$ does not see $a$,
      changing the color of $a$'s hat does not change $s$'s declaration.
      Thus $s$ makes the same declaration for $f_s$ and $f_a$,
      and hence $s$ is still incorrect for $f_a$.
      Since $f_a$ already has one incorrect prisoner, namely $s$,
      every prisoner other than $s$ must be correct under $P$.
      Therefore, for both $f_s$ and $f_a$,
      every prisoner other than $s$ who speaks before $a$ is correct.
      In summary, because $f_s$ and $f_a$ differ only in the color of $a$'s hat,
      neither the hats visible to $a$ nor the declarations of the other prisoners
      that $a$ hears before speaking, including the declaration of $s$, change.
      Thus $v_a^{f_s}=v_a^{f_a}$ and $h_a^{f_s}=h_a^{f_a}$, and hence
      \begin{equation*}
        P(f_s)(a)\ =\ \sigma_a(f_s)\ =\ \strategy{a}(h_a^{f_s},\ v_a^{f_s})\ =\ \strategy{a}(h_a^{f_a},\ v_a^{f_a})\ =\ P(f_a)(a).
      \end{equation*}
      This shows that $a$ makes the same declaration for $f_s$ and $f_a$.
      Since $a$ is correct for $f_s$,
      the construction of $f_a$ implies that $a$ is incorrect for $f_a$.
      Therefore both $s$ and $a$ are incorrect for $f_a$,
      contradicting $P\in \PS{\mathcal{G}}{e}{\leq}{1}$.

    \item[(\ref{sub:2})]

      Suppose that for some coloring $f$,
      there is a prisoner $a\in \SubPrSet{1-}$ other than $s$
      who guesses incorrectly under $P$.
      Since $P\in \PS{\mathcal{G}}{e}{\leq}{1}$,
      $s$ is correct for $f$.
      If we construct $f_s$ in the same way as in (\ref{sub:1}),
      then the same argument as in (\ref{sub:1}) shows that
      both $s$ and $a$ are incorrect for $f_s$.
      The existence of such a coloring $f_s$ contradicts
      $P\in \PS{\mathcal{G}}{e}{\leq}{1}$.

    \item[(\ref{sub:3})]

      Let $a\in\SubPrSet{1-}$ and $k\in K\setminus\{f(a)\}$ be arbitrary,
      and put $f_a=\cv{f}{a}{k}$.
      Suppose that $s$ makes the same declaration for $f$ and $f_a$.
      Then, arguing as in (\ref{sub:1}),
      neither the hats visible to $a$
      nor the declarations of the other prisoners that $a$ hears before speaking
      change between $f$ and $f_a$.
      Hence $a$ makes the same declaration for $f$ and $f_a$.
      On the other hand, the color of $a$'s own hat is different in $f$ and $f_a$,
      so $a$ must be incorrect for at least one of these colorings.
      This contradicts (\ref{sub:2}),
      which does not allow any prisoner other than $s$ to be incorrect.

    \item[(\ref{sub:4})]

      Fix $a\in\SubPrSet{1-}$ and put $k_s=h_a^f(s)$.
      Thus $k_s=P(f)(s)$ is $s$'s declaration under $P$ for $f$.
      Let $k\in K\setminus\{k_s\}$ be arbitrary.
      Suppose that even if $s$ declared $k$ instead of $k_s$,
      $a$'s own declaration would not change,
      provided that the information about the hats visible to $a$
      and the declarations heard from prisoners other than $s$ remained the same.
      That is,
      $\strategy{a}(h_a^f,v_a^f)
      =
      \strategy{a}( \cv{h_a^f}{s}{k},v_a^f )$.
      From this assumption, we construct a coloring in which $a$ guesses incorrectly. \\
      First define $\varphi:K\to K$ by
      $\varphi(x)=P(\cv{f}{a}{x})(s)$.
      This map assigns to each color $x$
      the color declared by $s$ when the color of $a$'s hat is changed to $x$.
      By (\ref{sub:3}), whenever the color of $a$'s hat changes,
      $s$'s declaration must also change,
      so $\varphi$ is injective.
      Since $K$ is finite, $\varphi$ is also surjective,
      and therefore bijective.
      In other words, for every color $c\in K$,
      there exists exactly one color of $a$'s hat
      that makes $s$ declare $c$. \\
      Choose $\ell,\dash{\ell}\in K$ so that
      $\varphi(\ell)=k_s$ and $\varphi(\dash{\ell})=k$.
      Since $\varphi$ is injective and $k_s\neq k$,
      we have $\ell\neq \dash{\ell}$.
      Furthermore, let
      $f_{\ell}=\cv{f}{a}{\ell}$ and
      $f_{\dash{\ell}}=\cv{f}{a}{\dash{\ell}}$.
      Then $f_{\ell}$ and $f_{\dash{\ell}}$ differ only in the color of $a$'s own hat,
      and $s$ declares $k_s$ and $k$, respectively. \\
      We now compare the information available to $a$ when $a$ speaks.
      Since only the color of $a$'s own hat is changed,
      the hats visible to $a$ remain the same, and
      $v_a^{f_{\ell}}=v_a^{f_{\dash{\ell}}}$.
      Moreover, every prisoner other than $s$ who speaks before $a$
      is correct for both colorings by (\ref{sub:2}).
      Since the colors of those prisoners' own hats do not change
      between $f_{\ell}$ and $f_{\dash{\ell}}$,
      their declarations do not change either.
      Therefore, among the information that $a$ hears before speaking,
      the only difference between $f_{\ell}$ and $f_{\dash{\ell}}$
      is $s$'s declaration,
      corresponding exactly to the change from $k_s$ to $k$. \\
      However, by our initial assumption,
      changing only $s$'s declaration from $k_s$ to $k$
      does not change $a$'s declaration.
      Hence $a$ makes the same declaration for
      $f_{\ell}$ and $f_{\dash{\ell}}$.
      On the other hand,
      the color of $a$'s own hat is different in the two colorings,
      since $\ell\neq\dash{\ell}$.
      Therefore $a$ is incorrect for at least one of
      $f_{\ell}$ and $f_{\dash{\ell}}$.
      This contradicts (\ref{sub:2}),
      which does not allow any prisoner other than $s$ to be incorrect.
      Therefore,
      whenever $s$'s declaration changes,
      $a$'s declaration must also change.
      \qedhere
  \end{itemize}
\end{proof}

\begin{theorem}[Continuation of Theorem \ref{thm:FFVA}]
  \label{thm:after-thm:FFVA}
  Suppose $K$ is finite with $3\le|K|<\omega$,
  $IN\ge2$,
  and $\mathcal{G}=\HatGame{A}{K}{V}{I}$ satisfies (S3).
  Then (b)$\Rightarrow$(a) holds:
  \begin{itemize}
    \item[(a)] $\mathcal{G}$ has (S1) and (S2).
    \item[(b)] $\PS{\mathcal{G}}{e}{\leq}{1}\neq\emptyset$. \qedhere
  \end{itemize}
\end{theorem}

\begin{proof}
  Assume (b), and fix a predictor
  $P\in\PS{\mathcal{G}}{e}{\leq}{1}$.

  First suppose that (S1) fails.
  Then $|\SubPrSet{0}|\ge2$.
  Since $A$ is finite, $\SubPrSet{0}$ is also finite.
  Consider the simultaneous game
  \[
    \mathcal{G}_1
    \;=\;
    \HatGame{\SubPrSet{0}}{K}{ V\cap(\SubPrSet{0}\times\SubPrSet{0}) }{1}.
  \]
  By Proposition \ref{prop:FirstGroup},
  there is a predictor in
  $\PS{\mathcal{G}_1}{e}{\leq}{1}$.

  Let $m=|\SubPrSet{0}|$.
  For any predictor in this simultaneous game,
  Lemma \ref{lem:average} says that the average number of correct guesses is
  $m/|K|$.
  On the other hand, since the predictor belongs to
  $\PS{\mathcal{G}_1}{e}{\leq}{1}$,
  at least $m-1$ prisoners guess correctly in every coloring.
  Hence the average number of correct guesses is at least $m-1$.
  Therefore
  \[
    \frac{m}{|K|}
    \;\ge\;
    m-1.
  \]
  But this is impossible when $m\ge2$ and $|K|\ge3$.
  This contradiction shows that (S1) must hold.
  Thus we may write $\SubPrSet{0}=\{s\}$.

  We next prove (S2).
  Suppose, toward a contradiction, that (S2) fails.
  Then there are prisoners $a,b\in A$ with $a\neq b$ such that
  $b\notin V(a)\cup H(a)$.
  Let $a\in\SubPrSet{n}$.
  Since $b\notin H(a)$, the prisoner $b$ does not declare before $a$.
  Hence $b\in\SubPrSet{n-}$.
  Also, since $b\notin V(a)$, we have $\lnot a\CanSee{V}b$.

  We claim that $n\neq0$.
  Indeed, since (S1) now holds, and since (S3) and
  $P\in\PS{\mathcal{G}}{e}{\leq}{1}$ also hold,
  Proposition \ref{prop:useful}(\ref{sub:1}) implies (S4).
  Thus the unique first speaker $s$ sees every other prisoner.
  Therefore $a$ cannot be this first speaker.
  Hence $n\ge1$.

  Since $b\in\SubPrSet{n-}$,
  either $I(a)=I(b)$ or $I(a)<I(b)$.
  We treat these two cases separately.

  First suppose that $I(a)=I(b)$.
  Take an arbitrary coloring $f\in\MapSet{A}{K}$, and define
  \[
    f_s
    \;=\;
    \cv{f}{s}{P(f)(s)}.
  \]
  Since $s$ sees all prisoners other than themselves by
  Proposition \ref{prop:useful}(\ref{sub:1}),
  the view of $s$ is the same under $f$ and $f_s$.
  Hence the declaration of $s$ is also the same:
  $P(f_s)(s)=P(f)(s)$.
  By the definition of $f_s$, this means that $s$ guesses correctly under $f_s$.

  Now choose a color different from $f_s(b)$ and define
  \[
    f_b
    \;=\;
    \cv{f_s}{b}{\nee{f_s(b)}}.
  \]
  The colorings $f_s$ and $f_b$ differ only at $b$.
  Since $b\neq s$, Proposition \ref{prop:useful}(\ref{sub:3}) gives
  $P(f_s)(s)\neq P(f_b)(s)$.
  But $f_b(s)=f_s(s)=P(f_s)(s)$.
  Therefore $P(f_b)(s)\neq f_b(s)$,
  so $s$ guesses incorrectly under $f_b$.

  On the other hand, under $f_s$ the prisoner $s$ guesses correctly.
  By Proposition \ref{prop:useful}(\ref{sub:2}),
  any incorrect guess must be made by $s$.
  Hence every prisoner other than $s$ guesses correctly under $f_s$.
  In particular, $P(f_s)(a)=f_s(a)$.

 Since $a$ does not see $b$, and $f_s$ and $f_b$ differ only at $b$,
  the visible coloring of $a$ is unchanged:
  $v_a^{f_s}=v_a^{f_b}$.
  By Proposition \ref{prop:useful}(\ref{sub:2}),
  every prisoner other than $s$ guesses correctly under both colorings.
  Since $b$ declares no earlier than $a$,
  all declarations heard by $a$ other than that of $s$ remain unchanged.
  Moreover, the declaration of $s$ changes from $f_s$ to $f_b$.
  Therefore, by Proposition \ref{prop:useful}(\ref{sub:4}),
  the declaration of $a$ must also change:
  $P(f_s)(a)\neq P(f_b)(a)$.

  Since $f_s(a)=f_b(a)$,
  Proposition \ref{prop:useful}(\ref{sub:2}) implies
  $P(f_s)(a)=P(f_b)(a)$,
  contradicting the inequality above.

  The case $I(a)<I(b)$ is treated in the same way.
  Indeed, the same construction of $f_s$ and $f_b$
  yields the same contradiction.

  In both cases we obtain a contradiction.
  Therefore (S2) must hold.
  Hence (a) holds.
\end{proof}

\begin{theorem}[Continuation of Theorem \ref{thm:FIVA}]
  \label{thm:after-thm:FIVA}
  Suppose $K$ is infinite,
  $IN\ge2$,
  and $\mathcal{G}=\HatGame{A}{K}{V}{I}$ satisfies (S3).
  Then (b)$\Rightarrow$(a) holds:
  \begin{itemize}
    \item[(a)] (S1) and (S4) hold.
    \item[(b)] $\PS{\mathcal{G}}{e}{\leq}{1}\neq\emptyset$. \qedhere
  \end{itemize}
\end{theorem}

\begin{proof}
  Arguing similarly,
  if $\SubPrSet{0}$ contains at least two prisoners,
  Proposition \ref{prop:FirstGroup}
  would yield a predictor with at most one error
  in the finite subgame restricted to prisoners who declare in the initial inning.
  On the other hand,
  Theorem \ref{thm:FIVanyS} forbids 
  guaranteeing at least one correct guess 
  (equivalently forbids guaranteeing at most $|\SubPrSet{0}|-1$ errors) 
  among finitely many prisoners
  if $K$ is infinite.
  Hence (S1) must hold,
  that is, $|\SubPrSet{0}|=1$.
  Then (S4) follows 
  from Proposition \ref{prop:useful}(\ref{sub:1}) 
  whenever (S1), (S3),
  and $\PS{\mathcal{G}}{e}{\leq}{1}\neq\emptyset$ hold.
\end{proof}

We summarize the results of this section according to the number of colors:
\begin{itemize}
  \item 2 colors:
  as a special phenomenon, a predictor with at most one error can exist
  under the structural conditions (S4), (S5), and (S6)
  (Theorem \ref{thm:F2VA}).
  Here (S5) allows the initial inning to consist of exactly two prisoners,
  unlike the one-first-speaker situation in the results for three or more finite colors
  and infinitely many colors.
  A complete characterization of finite non-simultaneous two-color puzzles
  admitting such a strategy remains a subject for future research.

  \item 3 or more finite colors:
  a predictor with at most one error exists if (S1) and (S2) hold
  (Theorem \ref{thm:FFVA}).
  Moreover, under the additional assumption (S3),
  these conditions are also necessary
  (Theorem \ref{thm:after-thm:FFVA}).

  \item infinitely many colors:
  a predictor with at most one error exists if (S1) and (S4) hold
  and there exists a bijection between $K$ and $\MapSet{A\setminus\{s\}}{K}$
  (Theorem \ref{thm:FIVA}).
  In particular, this holds under the Axiom of Choice (Corollary \ref{cor:FIVA-AC}).
  Conversely, under the additional assumption (S3),
  the existence of such a predictor implies (S1) and (S4)
  (Theorem \ref{thm:after-thm:FIVA}).
\end{itemize}

\section{Non-simultaneous puzzle with infinite prisoners}
\label{section:InfinitePrisoner}

In this section, assume $A$ is infinite.
The notions of a parity function and the Axiom of Choice play an important role here. 
The relationship between parity functions, the Axiom of Choice, and hat puzzles has been studied in \cite{Hat002}, 
and further related work can be found in \cite{Hat073}.

We extend the idea of Lemma \ref{lem:FGVA} to show that,
if a suitable ``infinite analog'' of summation exists,
then we can ensure at most one prisoner guesses incorrectly.
Since conventional summation does not directly apply to infinitely many hats,
we replace it with a ``parity function'':

\begin{definition}[]
  Let $X$ be a set,
  $(G,+)$ an abelian group,
  and consider $\map{\phi}{\MapSet{X}{G}}{G}$.
  We say $\phi$ is an $\MapSet{X}{G}$-parity function if for all $f\in \MapSet{X}{G}$,
  $x\in X$,
  and $g_1,g_2\in G$,
  \[
  \phi\bigl(\cv{f}{x}{g_1}\bigr)
  \;-
  \phi\bigl(\cv{f}{x}{g_2}\bigr)
  \;=\;
  g_2-g_1.
\]
  In other words,
  changing one coordinate of $f$ from $g_1$ to $g_2$ 
  shifts the value of $\phi(f)$ by $g_1-g_2$ in a uniform way.
\end{definition}

This notion follows \cite{Hat002} and \cite{Hat006},
though we employ a version generalized in \cite{Hat006}.
The next theorem is the infinite analogue of Lemma \ref{lem:FGVA}.
In Lemma \ref{lem:FGVA}, the first prisoner gives the sum of all visible colors,
and each later prisoner recovers their own color by subtracting from this sum
the information already heard and the information still visible.
In the infinite setting, the ordinary finite sum is no longer available,
and it is replaced by a parity function.

\begin{theorem}[]
  \label{thm:IGVA}
  Let $A$ be infinite and $(K,+)$ be an abelian group.
  Suppose there is an $\MapSet{A}{K}$-parity function $\phi$.
  Then in any game $\mathcal{G}=\HatGame{A}{K}{V}{I}$ with $IN\ge2$ satisfying
  (S1) and (S2), 
  there is a predictor in $\PS{\mathcal{G}}{e}{\leq}{1}$
  ensuring at most one error in every coloring.
\end{theorem}

\begin{proof}
  Let $0$ be the identity element of $(K,+)$,
  and let $\phi:\MapSet{A}{K}\to K$
  be an $\MapSet{A}{K}$-parity function.
  By (S1), write $\SubPrSet{0}=\{s\}$.

  The first prisoner $s$ plays the role of a hint-giver.
  Since $s$ sees all hats except their own,
  prisoner $s$ temporarily assigns the value $0$ to themselves
  and declares the parity value of the resulting coloring.
  Thus, for $f_s\in \MapSet{V(s)}{K}$, define
  \[
    \strategy{s}(f_s)
    \;=\;
    \phi\Bigl(f_s\cup\bigl\{\op{s}{0}\bigr\}\Bigr).
  \]

  A later prisoner $a$ uses the declaration of $s$ as the hint.
  To recover their own color,
  prisoner $a$ forms the coloring obtained by putting $0$ at $s$ and at $a$,
  using the declarations already heard from earlier prisoners,
  and using the visible colors of the remaining prisoners.
  Applying $\phi$ to this coloring and subtracting the declaration of $s$
  recovers the color of $a$.
  Formally, for $a\in \SubPrSet{1-}$,
  and for two colorings $f_a\in \MapSet{V(a)}{K}$ and $h_a\in \MapSet{H(a)}{K}$,
  define
  \[
    \strategy{a}(h_a,f_a)
    \;=\;
    \phi\Bigl(\;
      \cv{h_a}{s}{0}
      \;\cup\; \bigl\{\op{a}{0}\bigr\}
      \;\cup\; f_a\restric\bigl(\dom(f_a)\setminus\dom(h_a)\bigr)
    \;\Bigr)
    \;-\; h_a(s).
  \]
  Let $P$ be the predictor generated by the family
  $\{\sigma_s\}\cup\{\sigma_a\}_{a\in\SubPrSet{1-}}$.

  As in Lemma \ref{lem:FGVA}, the correctness of
  the strategy follows by transfinite induction
  on the innings, using the defining property
  of the parity function in place of finite summation.
  Thus every prisoner other than $s$ guesses correctly,
  and hence $P\in\PS{\mathcal{G}}{e}{\leq}{1}$.

\end{proof}

The remaining question is whether such a parity function exists.
If we assume the Axiom of Choice,
it is straightforward to build one.
The construction follows the ``pick a representative in each $=^*$-equivalence class'' approach from \cite{Hat002} (Lemma 6).

\begin{lemma}[]
  \label{lem:AC-PF}
  Under the Axiom of Choice,
  for any infinite set $X$ and abelian group $(G,+)$,
  there is an $\MapSet{X}{G}$-parity function.
\end{lemma}

\begin{proof}
  Consider the equivalence relation $=^*$ on $\MapSet{X}{G}$:
  $f=^{*}g$ iff $f$ and $g$ differ only at finitely many points of $X$.
  Using Choice, pick a selector $\varphi$ 
  so that $\varphi([f])\in[f]$.
  Define $\phi(f)$ 
  by summing up $\varphi([f])(x)-f(x)$ 
  over those $x$ where $f(x)\neq \varphi([f])(x)$ (only finitely many).
  One verifies $\phi$ is a parity function by checking that 
  altering one coordinate changes $\phi(f)$ 
  by exactly the difference in the group values.
\end{proof}

Hence whenever $(K,+)$ is an abelian group,
$A$ is infinite, and an $\MapSet{A}{K}$-parity function exists,
we can construct a predictor with at most one error for any game satisfying (S1) and (S2).
Under the Axiom of Choice, every set $K$ can be equipped
with an abelian group structure \cite[Corollary~(f)]{Logic012}.
Together with the existence of a parity function (Lemma \ref{lem:AC-PF}),
we can apply Theorem \ref{thm:IGVA} to obtain:

\begin{corollary}[]
  \label{cor:IAVA}
  Assume the Axiom of Choice. Let $A$ be infinite,
  $K$ any set,
  and consider a game $\mathcal{G}=\HatGame{A}{K}{V}{I}$ 
  with $IN\ge2$ satisfying (S1) and (S2).
  Then $\PS{\mathcal{G}}{e}{\leq}{1}\neq\emptyset$,
  that is, there is a predictor 
  with at most one error in every coloring.
\end{corollary}

\section{Two infinite prisoner puzzles}
\label{section:TwoTypeInInfinitePuzzle}

Both in simultaneous and non-simultaneous puzzles with infinite prisoners,
the best-known results for guaranteeing 
only finitely many errors use the Axiom of Choice.
Specifically,
Theorem \ref{thm:FEP} yields a finite error predictor 
for the complete-visibility,
simultaneous case. 
Corollary \ref{cor:IAVA} yields a predictor with at most one error 
for the non-simultaneous case under conditions 
that allow a ``hint'' speaker.
The two constructions use the Axiom of Choice differently.
The proof of Theorem \ref{thm:FEP} uses Choice to select representatives
of equivalence classes, without requiring an abelian group structure on $K$.
In contrast, the construction underlying Corollary \ref{cor:IAVA} uses Choice
both to endow $K$ with an abelian group structure (if necessary)
and to construct a parity function.
In fact, the use of Choice to construct a parity function in the non-simultaneous case
can be replaced by a certain property of a predictor from the simultaneous puzzle.
The relevant property is ``robustness'',
meaning the predictor's output is insensitive to changes 
on finitely many prisoners' hats:

\begin{definition}
  For $\mathcal{G}=\HatGame{A}{K}{V}{1}$,
  say a predictor $P$ is robust if $f=^{*}g$ implies $P(f)=P(g)$.
  That is,
  finite changes to $f$ do not alter the collective guesses $P(f)$.
  Denote the set of robust predictors by $\robustPS{\mathcal{G}}$.
\end{definition}

Inspecting the construction in the proof of Theorem \ref{thm:FEP}, 
we see that the finite error predictor obtained there is in fact robust. 
This gives the following result.

\begin{theorem}[Gabay--O'Connor]
  Under the Axiom of Choice,
  for game $\mathcal{G}=\HatGame{A}{K}{V}{1}$ with infinite $A$ and complete $V$,
  there is a finite error predictor that is also robust,
  that is, in $\FEPS{\mathcal{G}}\cap\robustPS{\mathcal{G}}$.
\end{theorem}

Furthermore,
a robust finite error predictor yields a parity function:

\begin{theorem}[]
  \label{thm:robust-FEP-parity}
  Suppose $A$ is infinite and $(K,+)$ is an abelian group. 
  If there is a robust finite error predictor 
  in $\mathcal{G}=\HatGame{A}{K}{V}{1}$, 
  then there is an $\MapSet{A}{K}$-parity function.
\end{theorem}

\begin{proof}
  Let $P\in \FEPS{\mathcal{G}}\cap\robustPS{\mathcal{G}}$.
  For $f\in \MapSet{A}{K}$,
  define
  \[    \phi(f)=\sum_{a\colon f(a)\neq P(f)(a)} \bigl[P(f)(a)-f(a)\bigr].   \]
  By the robustness and finite error property,
  changing one prisoner's hat color changes $\phi(f)$ consistently,
  letting us prove the parity condition 
  in exactly the same manner as Lemma \ref{lem:AC-PF}.
\end{proof}

By combining Theorem \ref{thm:IGVA} with Theorem \ref{thm:robust-FEP-parity},
we can prove, without Choice,
that 
if the simultaneous game on $(A,K)$ admits a robust finite error predictor,
then the non-simultaneous version with (S1) and (S2) 
also admits a predictor with at most one error. 

\begin{corollary}[]
  Let $A$ be infinite and $(K,+)$ an abelian group.
  If the simultaneous game $\mathcal{G}_1=\HatGame{A}{K}{V}{1}$ 
  has a predictor in $\FEPS{\mathcal{G}_1}\cap\robustPS{\mathcal{G}_1}$,
  then for the non-simultaneous game 
  $\mathcal{G}=\HatGame{A}{K}{V}{I}$ with $IN\ge2$ satisfying (S1) and (S2),
  we can construct a predictor in $\PS{\mathcal{G}}{e}{\leq}{1}$.
\end{corollary}

\section{Questions}
\label{section:questions}

This paper has focused on hat puzzles 
in which prisoners do not declare simultaneously.

For the finite case,
we saw that the existence of a ``best possible''
predictor (either always at least one correct or at most one error) 
once again depends on the number of colors relative 
to the number of prisoners,
mirroring the simultaneous scenario.

When there are two colors and the puzzle is simultaneous,
it is known (for example, Theorem \ref{thm:F2VcyclicS}) 
that having a cyclic visibility graph is equivalent to guaranteeing 
at least one correct guess.
\begin{question}[]
  \label{q:1}
  Lemma \ref{lem:FGVA} establishes the existence of a predictor with at most one error
  under conditions (S1) and (S2), while Theorem \ref{thm:F2VA} establishes such existence
  under conditions (S4), (S5), and (S6).

  For finite non-simultaneous two-color hat games, can we give necessary and sufficient
  conditions on the structure of the game for the existence of a predictor with at most one error?
\end{question}

For finite prisoner sets, Theorems \ref{thm:FFVA} and \ref{thm:after-thm:FFVA}
characterize the existence of a predictor with at most one error
in the finite $|K|\ge3$ case under (S3).
For infinitely many colors, Theorems \ref{thm:FIVA} and \ref{thm:after-thm:FIVA}
give the corresponding characterization under (S3),
provided the bijection assumed in Theorem \ref{thm:FIVA} exists.
In particular, this characterization holds under the Axiom of Choice.
However, in these proofs,
(S3) was assumed to demonstrate the converse direction.
We ask whether (S3) can be eliminated:

\begin{question}[Continuation of Theorem \ref{thm:after-thm:FFVA} 
and Theorem \ref{thm:after-thm:FIVA}]
  In Theorems \ref{thm:after-thm:FFVA} and \ref{thm:after-thm:FIVA},
  is (S3) essential for proving (b)$\Rightarrow$(a)?
  Or can we remove (S3) and still obtain these implications?
\end{question}

Finally,
for a non-simultaneous puzzle with infinite prisoners,
we relied on the existence of a parity function for the color group 
to guarantee that all but possibly one prisoner guess correctly.
It is still open whether the existence of such a parity function is also necessary,
or whether one might need an analogue of (S3) or similar constraints.

\begin{question}[Analogue of Theorem \ref{thm:after-thm:FFVA} in the infinite setting]
  Let $A$ be infinite and $(K,+)$ an abelian group. Consider the following statements:
  \begin{itemize}
    \item[(a)] There is an $\MapSet{A}{K}$-parity function.
    \item[(b)] 
      If $\mathcal{G}=\HatGame{A}{K}{V}{I}$ 
      with $IN\ge2$ satisfies
      (S1) and (S2),
      then $\PS{\mathcal{G}}{e}{\leq}{1}\neq\emptyset$.
  \end{itemize}
  We have shown (a)$\Rightarrow$(b).
  Is (b)$\Rightarrow$(a)$\,$ true?
  Perhaps an additional condition akin to (S3) is required,
  or some other structural assumption?
\end{question}

In the special case $A=\omega$ and $K=2$,
Lemma 5 of \cite{Hat002} shows (a) and (b) are equivalent.
We suspect that for more general $A,K$,
we 
may need an assumption analogous to (S3)  
to make (b)$\Rightarrow$(a) hold.

\section*{Acknowledgments}
The author would like to express his sincere gratitude to his supervisor,
Masaru Kada, for his invaluable guidance and support,
both in the preparation of this paper and more broadly throughout the author's academic work.

The author used the assistance of ChatGPT (OpenAI) for checking the English grammar and improving the readability of certain sentences.
All mathematical arguments, proofs, and results were created independently by the author.

\bigskip

\begin{quote}
	\textsc{Souji SHIZUMA}\\
	Graduate School of Science, Osaka Prefecture University\\
	3--3--138 Sugimoto, Sumiyoshi Osaka-shi 558--8585 JAPAN\\
	\texttt{gsouji0426@gmail.com}
\end{quote}

\end{document}